\documentclass[11pt,reqno]{amsart}
\usepackage[foot]{amsaddr}
\usepackage{amsthm,amssymb,amsfonts}
\usepackage{amsmath}
\usepackage{amscd} 
\usepackage{setup}
\usepackage{microtype}

\usepackage[normalem]{ulem}
\usepackage[all]{xy}

\newcommand{\N}{\mathbb{N}}

\newcommand{\R}{\mathbb{R}}
\newcommand{\Z}{\mathbb{Z}}
\usepackage{bbm}
\renewcommand{\1}{\mathbf{1}}

\newcommand{\NN}{\mathcal{N}}

\newcommand{\ad}{\text{ and }}

\newcommand{\ceq}{\coloneqq} 

\DeclarePairedDelimiter\absd{\lvert}{\rvert} 

\newcommand{\bpar}[1]{\left(#1\right)}

\newcommand{\bs}{\baselineskip}

\newcommand{\ind}[1]{\1_{[#1]}} 
\newcommand{\disteq}{\stackrel{\mathrm d}{=}} 

\newcommand{\E}[1]{{\mathbf E}\left[#1\right]}				
\newcommand{\va}{{\mathbf{Var}}}

\newcommand{\Ber}{\mathrm{Ber}}

\title[Asymptotics for harmonic descent chain and applications]{Asymptotics for the harmonic descent chain and applications to critical beta-splitting trees}

\author{Anna Brandenberger \and Byron Chin \and Elchanan Mossel}

\address{Department of Mathematics, MIT}
\email{\{abrande,byronc,elmos\}@mit.edu}

\begin{document}

\begin{abstract}
Motivated by the connection to a probabilistic model of phylogenetic trees introduced by Aldous, we study the recursive sequence governed by the rule $x_n = \sum_{i=1}^{n-1} \frac{1}{h_{n-1}(n-i)} x_i$ where $h_{n-1} = \sum_{j=1}^{n-1} 1/j$, known as the harmonic descent chain. While it is known that this sequence converges to an explicit limit $x$, much remained unknown about the rate of convergence. We first show that a class of recursive sequences including the above are decreasing and use this to bound the rate of convergence. Moreover, for the harmonic descent chain we prove the asymptotic $x_n - x = n^{-\gamma_* + o(1)}$ for an implicit exponent $\gamma_*$. Implementing a variant of the contraction method with degenerate limit equation, we deduce central limit theorems for various statistics of the critical beta-splitting random tree. This answers a number of questions of Aldous, Janson, and Pittel.
\end{abstract}

\maketitle 

\section{Introduction}

The critical beta-splitting tree is a probabilistic model for phylogenetic trees, introduced by David Aldous in \cite{aldous1996probability} and subject to extensive recent study as surveyed in \cite{aldous2025critical}. In this paper, we answer some questions of Aldous and Janson given in their survey on the topic~\cite{aldous2023critical}. Namely, we answer Open Problems 5 and 6, and partially answer Open Problem 8 by proving asymptotics for a process known as the harmonic descent chain and deducing central limit theorems for various statistics of the tree; see section~\ref{subsec: defs} for definitions. These questions essentially involve recursions defined as follows: 
\begin{equation}\label{eq:recursion}
x_n = \sum_{i=1}^{n-1} \frac{1}{h_{n-1}(n-i)} \cdot x_i \quad \text{for $n > k$}    
\end{equation}
with initial conditions $x_1, \ldots, x_{k-1} = 0$, $x_k > 0$, where $h_{n-1} = \sum_{j=1}^{n-1} 1/j$ is the harmonic sum. Our main theorem gives bounds on the rate of convergence for this recursive sequence.

\begin{thm}\label{thm:exponent}
    Let $k \ge 2$. We have that 
    \begin{enumerate}[ref=(\roman*)]
        \item \label{thm:x_n-decreasing} $x_{n+1} < x_n$ for all $n \geq k$. In particular, for any fixed $k$, the limit $x \ceq \lim_{n \to\infty} x_n$ exists.
        \item \label{thm:x_n-simple-rate} For all $n > k$, $x_n - x \leq \frac{2}{\log 2}\cdot\frac{kx_k}{n-k}$.
        \item \label{thm:bootstrap}
        For fixed $c > 2$ and $\ell \in \N$, and for $k$ sufficiently large, $x_{k+m} \leq \frac{c^\ell x_k}{mh_k}$ for each $m \leq k^{d_\ell}$ where $d_\ell = 1-\frac{1}{c+1} - 2^{-\ell}$.

        \item \label{thm:exponent-fixed-k} For fixed $k \in \N$, $x_n - x = n^{-\gamma_*+o(1)}$, where $\gamma_*$ is the solution to 
        \begin{equation}
        \sum_{i=1}^\infty \left(\frac{1}{i} - \frac{i}{(i+1)(i+1-\gamma)}\right) = 0 
        \end{equation}
        with $1.5 < \gamma_* < 2$.
        \item \label{thm:exponent-growing-k} For all $\epsilon > 0$ and $k = o(n/\log^3 n)$, there exist $0 < c(k, \epsilon) < C(k, \epsilon) = O_\epsilon(x_k k^{0.01})$ such that $c(k, \epsilon)(n / k)^{-\gamma_*-\epsilon} \leq x_n - x \leq C(k, \epsilon)(n / k)^{-\gamma_*+\epsilon}$.
    \end{enumerate}
\end{thm}
Note that to simplify notation, we suppress the dependence on $k$, writing $x_n \ceq x_n(k)$ and $x \ceq x(k)$ throughout.
\begin{rem}
    The fact that $\lim_{n \to \infty} x_n$ exists was already shown in \cite{aldous2025critical}. In fact, an explicit form of the limit is known \cite{aldous2024harmonic, iksanov2025harmonic, aldous2024critical3, aldous2024critical4}.
\end{rem}
\begin{rem}
    Numerically, $\gamma_* = 1.567353753101655332547...$.
\end{rem}
\begin{rem}\label{remark}
    For concreteness, we note here that we use~\ref{thm:bootstrap} with the following instantiations:
    \begin{itemize}
        \item $c=3,\ \ell=1$: $x_{k+m} \leq \frac{3x_k}{mh_k}$ for $m \leq k^{1/6}$. 
        \item $c=1000,\ \ell=9$: $x_{k+m} = O(\frac{x_k}{mh_k})$ for $m \leq k^{0.995}$. In particular, since $k^{0.995} \geq n^{0.99}$ when $k \geq n-n^{0.99}$, we have $x_n \leq \frac{x_k}{n-k}$ for such $k$ and $n$ sufficiently large.
    \end{itemize}
\end{rem}

Let us briefly discuss the interpretation of the above theorem. Item~\ref{thm:x_n-decreasing} establishes the useful decreasing property of the sequence conjectured by Aldous and Pittel~\cite{aldous2025critical} and reproves the existence of a limit \cite[Theorem 2.16]{aldous2025critical}. Items~\ref{thm:x_n-simple-rate} and~\ref{thm:bootstrap} are preliminary long-term (large $n$) and short-term (small $n$) error estimates, respectively. Their proofs are simple and suffice for some applications. Items~\ref{thm:exponent-fixed-k} and~\ref{thm:exponent-growing-k} are sharper error rates that are required for more delicate applications. In particular,~\ref{thm:exponent-fixed-k} is a precise asymptotic for the exponent of the main error term. 

Theorem~\ref{thm:exponent}\ref{thm:x_n-decreasing} follows from a more general statement about the behavior of recursive averages. Notably, the conditions apply to $\beta$-splitting trees for any $\beta < 0$ (defined below), for which $\frac{p(n,i)}{p(n+1,i)}$ is a constant multiple of $\frac{n+1-i}{n+1-i+\beta}$.
\begin{prop}\label{prop:x_n-decreasing}
    Let $k \ge 2$. Consider the sequence $(x_i)_{i=1}^\infty$ with $x_1, \ldots, x_{k-1} = 0, \ x_k > 0$ and 
    \[ x_n = \sum_{i=1}^{n-1} p(n, i)x_i \ \text{ for $n > k$}, \]
    where $p(n,i) > 0$ for every $n > k$ and $1 \le i \le n-1$, $\sum_{i=1}^{n-1} p(n,i) = 1$ and $\frac{p(n,i)}{p(n+1, i)}$ is non-decreasing in $i$ for every $n$. Then $x_{n+1} < x_n$ for all $n \geq k$. In particular, $\lim_{n \to \infty} x_n$ exists. 
\end{prop}

\subsection{Beta-splitting tree model}\label{subsec: defs}
We now introduce the critical beta-splitting tree, which is the subject of our main application for Theorem~\ref{thm:exponent}. For any $m \geq 2$, define the distribution $(q_{m}(i))_{i=1}^{m-1}$ as 
\begin{equation}\label{eq:qmi}
    q_m(i) = \frac{m}{2h_{m-1}} \frac{1}{i(m-i)},
\end{equation}
where $h_{m-1} = \sum_{j=1}^{m-1} 1/j$ is the harmonic sum. The critical beta-splitting tree DTCS($n$) is a random binary tree with $n$ leaves (labelled $[n]$) built as follows. First, split $[n]$ into left and right subtrees $\{1, \dots, L_n\}$ and $\{L_n + 1, \dots, n\}$ of respective sizes $L_n$ and $R_n \ceq n-L_n$, where $L_n$ (and $R_n$) has distribution $(q_n(\cdot))$. Then, recursively split each interval of size $m \geq 2$ into two subtrees according to $q_m(\cdot)$, stopping when $m=1$. 
We note that the general $\beta$-splitting model is defined similarly with the probability vector
\[ q_m^{\beta}(i) \propto \frac{\Gamma(\beta+i+1)\,\Gamma(\beta+m-i+1)}{\Gamma(i+1)\,\Gamma(m-i+1)}, \qquad 1 \leq i \leq m-1, \]
for $\beta > -2$~\cite{aldous1996probability}, which is comparable to $(i(m-i))^{\beta}$ and reduces to \eqref{eq:qmi} for $\beta = -1$.
The choice $\beta = -1$ is deemed \textit{critical} due to leaf-heights changing from order $n^{-\beta-1}$ below this value to order $\log n$ above it \cite{aldous1996probability}.

A substructure of interest in these models are \textit{fringe subtrees}. A fringe subtree is a subtree induced by picking all descendants of one of the internal vertices of the tree. We refer to these as \textit{clades}, inspired by the phylogenetic motivation of the model. 
 
Aldous and Pittel~\cite{aldous2025critical} also introduce a continuous time version of the model, denoted CTCS($n$), where an interval of size $m \geq 2$ splits at rate $h_{m-1}$ into two intervals. Then, for a uniformly random leaf $\ell$, we can consider the discrete-time Markov chain describing the sequence of clade sizes on the path from the root to $\ell$. The continuous version which exits a clade of size $m$ at rate $h_{m-1}$ is a continuous-time Markov chain on $\Z_{>0}$ with transition rates 
\[
\lambda_{m,i} = (m-i)^{-1}, \ 1 \leq i \leq m-1,\ m \geq 2,
\]
with 1 absorbing. This is known as the harmonic descent chain. Let $a(n,k)$ denote the occupation probability, i.e., the probability that the chain started at state $n$ is ever in state $k$, with $a(n,n) = a(n,1)= 1$. In \cite[Theorem 18]{aldous2023critical}, the authors state that the limit $a(k) \ceq \lim_{n\to\infty} a(n,k)$ exists, and equals
\begin{equation}\label{eq:a_k}
a(k) = \frac{6h_{k-1}}{\pi^2 (k-1)}    
\end{equation}
for $k \ge 2$.
Various proofs are known: \cite{aldous2024harmonic} via coupling and ``inspired guesswork", \cite{iksanov2025harmonic} via regenerative composition structures, and \cite{aldous2024critical3, aldous2024critical4} via exchangeable representations. \cite[Open Problem 8]{aldous2023critical} asks about more detailed rates of convergence.

Note that the sequence $a(n,k)/k$ has the form considered in \eqref{eq:recursion}. Indeed, letting $N_n(k)$ be the number of clades of size $k$ in DTCS$(n)$, we can write $a(n,k)$ in terms of $\E {N_n(k)}$ exactly as
\vspace{-.5\baselineskip}
\[
    \E {N_n(k)} = \frac{n a(n,k)}{k},
\]
with $\E{N_n(k)} = 0$ for $n < k$ and $\E {N_k(k)} = 1$.
By the law of total expectation, for $n > k$, we have 
\begin{align}
    \frac{\E {N_n(k)} }{n} &= \frac{1}{n} \sum_{i=1}^{n-1} q_{n}(i) (\E{N_i(k)} + \E{N_{n-i}(k)}) \nonumber \\ 
    &= \sum_{i=1}^{n-1} \frac{1}{2h_{n-1} i (n-i)} \E {N_i(k)} + \sum_{i=1}^{n-1} \frac{1}{2h_{n-1} i (n-i)} \E {N_{n-i}(k)} \nonumber \\ 
    &= \sum_{i=1}^{n-1} \frac{1}{h_{n-1} (n-i)} \frac{\E{N_i(k)}}{i}. \label{eq:total-expectation}
\end{align}
Therefore, from Theorem~\ref{thm:exponent}\ref{thm:exponent-fixed-k}--\ref{thm:exponent-growing-k}, for fixed $k \ge 2$,
\begin{equation}
    a(n,k) - a(k) = n^{-\gamma_* + o(1)}
\end{equation}
and generally, for $k = o(n/\log^3 n)$, there exist $0 < c(k,\epsilon) < C(k,\epsilon) = O_\epsilon(k^{0.01})$ such that
\begin{equation}
    c(k,\epsilon) (n/k)^{-\gamma_* - \epsilon} \leq {a(n,k) - a(k)} \leq  C(k,\epsilon) (n/k)^{-\gamma_* + \epsilon}.
\end{equation}
This gives an answer to \cite[Open problem 8]{aldous2023critical} for fixed $k$ and makes progress for $k$ growing with $n$. Moreover, a similar analysis as the proof of Corollary~\ref{cor:CLT-length} below proves the following case of~\cite[Ansatz 20]{aldous2023critical}:
\[ \sum_{k=2}^n a(n,k)f(k) \to \sum_{k=2}^\infty a(k)f(k) \]
for any non-negative sequence $(f(k))_{k \geq 2}$ satisfying $f(k) = O(k^{-\epsilon})$ for some $\epsilon > 0$.

\subsection{A few central limit theorems}
Using our main technical theorem, we apply a framework for general limit law results for random variables satisfying distributional recursions to show central limit theorems for several quantities of interest. Specifically, we prove the following variant of the contraction method with degenerate limit equation \cite[Theorem 5.1]{NR:04} to answer \cite[Open Problems 5 and 6]{aldous2023critical}.

\begin{prop}\label{prop: contraction}
    Let $(X_n')_{n \geq 1}$ and $(X_n'')_{n \geq 1}$ be independent copies of $(X_n)_{n \geq 1}$, independent also of $(I_n, b_n)_{n \geq 2}$, where $I_n = \max\{L_n, R_n\}$, and suppose that
    \[ X_n \disteq X'_{I_n} + X''_{n-I_n} + b_n \quad \text{for } n \geq 2. \]
    Suppose further that there exist constants $\mu, \sigma, \epsilon, C > 0$ such that $\absd{\E{X_n} - \mu n} \leq Cn^{1/2-\epsilon}$ and $\absd{\va[X_n] - \sigma^2 n} \leq Cn^{1-\epsilon}$, and $\E{|b_n|^3} \leq n^{3/2-\epsilon}$. Then
    \[ \frac{X_n - \E{X_n}}{\sqrt{\va[X_n]}} \xrightarrow{\mathrm d} \mathcal{N}(0,1). \]
\end{prop}
A few remarks are in order. The usual contraction method \cite{neininger2004general} does not apply as our limit equation turns out to be degenerate. With our choice of beta-splitting distribution \eqref{eq:qmi}, $I_n$ is such that $\sqrt{{I_n}/{n}} \to 1$ and $\sqrt{{(n-I_n)}/{n}} \to 0$ in probability; in particular, $I_n$ always captures the bulk of the variance. 
Typically, this degeneracy occurs in the case where $X_n$ has poly-logarithmic variance. In our case, we have linear mean and variance, and the degeneracy arises from the harmonic splitting of the tree. Thus, the usual degenerate form of the contraction method \cite{NR:04} does not apply directly either. 
This difficulty involving degeneracy for the critical beta-splitting tree was observed by Kolesnik~\cite{kolesnik2025critical}, whose work adapts the contraction method to prove a central limit theorem for the height of a random leaf in the tree. His use of contraction relies on the height depending on only one side of the split and the variance being poly-logarithmic, so it is not applicable to our setting either. Nevertheless, we are also able to adapt the contraction method to our setting; in section~\ref{subsec:proof-ideas} below, we discuss the differences in our application.

For the discrete model, in addition to $N_n(k)$ the number of clades of size $k$ in DTCS$(n)$, we can, for a fixed clade $\chi$, i.e., a fixed fringe subtree $\chi$, define $N_n(\chi)$ to be the number of copies of $\chi$ in DTCS($n$). We only require~\ref{thm:x_n-decreasing}--\ref{thm:bootstrap} from Theorem~\ref{thm:exponent} to resolve the following. 

\begin{cor}[Central limit theorem for $N_n(k)$ and $N_n(\chi)$] \label{cor:CLT-fringe}
For any $k \ge 2$, we have
    \[ \frac{N_n(k) - \E{N_n(k)}}{\sqrt{\va[N_n(k)]}} \xrightarrow{\mathrm d} \NN(0,1). \]
A central limit theorem also holds for $N_n(\chi)$ for all $\chi$ of size $|\chi| \ge 2$.
In both cases the limit $\sigma^2 = \lim_{n \to \infty} \frac1n \va[N_n(\cdot)]$ exists and is strictly positive; see Remark~\ref{rem: positive variance}.
\end{cor}

For the continuous model, let $\Lambda_n$ denote the total length of CTCS$(n)$, where the length of an edge above a clade of size $m$ corresponds to the duration of time Exp($h_{m-1}$) before it splits (for an illustration, see \cite[Figure~3]{aldous2023critical}). It is known that $\lim_{n\to\infty} n^{-1} \E {\Lambda_n} = 6/\pi^2$.

\begin{cor}[Central limit theorem for $\Lambda_n$]\label{cor:CLT-length}
    We have
    \[ \frac{\Lambda_n - \E{\Lambda_n}}{\sqrt{\va[\Lambda_n]}} \xrightarrow{\mathrm d} \NN(0,1). \]
\end{cor}

\subsection{Proof ideas} \label{subsec:proof-ideas}
We take a hands-on approach to proving Theorem~\ref{thm:exponent}. The key transformation is to turn the recursive relation on the sequence $(x_n)_{n \in \N}$ into a relation on the consecutive differences $(x_n - x_{n+1})_{n \in \N}$. We show inductively that these consecutive differences are always positive. This proves~\ref{thm:x_n-decreasing}.  Using the explicit form of the recursive relation on consecutive differences, we also deduce~\ref{thm:x_n-simple-rate} and~\ref{thm:bootstrap} inductively. 

To prove parts~\ref{thm:exponent-fixed-k} and~\ref{thm:exponent-growing-k}, we approximate the recursive relation on consecutive differences by an integral. With the assumption that $x_n - x_{n+1} \sim n^{-\gamma}$, this integral yields an equation governing $\gamma$, which is precisely the equation defining the exponent $\gamma_*$. A key difficulty in this step is that the integral is divergent. The equation arises as a result of comparing the divergent integral to the growth of $h_n$. This necessitates understanding the error of the integral approximation, which could a priori be diverging as well. The upper and lower bounds are once again proven inductively with the help of the integral approximation. 

Our central limit theorems follow from the heuristic that if two random variables are close to Gaussian, then their sum is closer to Gaussian. Inductively, this implies that a recursively defined statistic of the beta-splitting tree is close to a mixture of Gaussian distributions. In order to obtain a central limit theorem, we need to check that the mean and variance of the statistic both scale linearly, so that the mixture of Gaussians also follows a Gaussian distribution. The analysis of the harmonic descent chain verifies this for several variables of interest. To formalize the argument we rely on a general strategy of Neininger and R\"uschendorf \cite{NR:04} which captures this heuristic. 

Compared to the work of Kolesnik~\cite{kolesnik2025critical}, our argument differs in one primary aspect. The height of a random leaf in the tree has poly-logarithmic variance, so his main bottleneck arises from computing the mean and variance to sufficient precision. The linear mean and variance of our random variables sidestep this difficulty. However, while the height of the tree depends only on a single side of the split, our random variables depend on both sides of the split, which increases the complexity of the recurrence. Our difficulty arises primarily from handling the contributions from both sides of the split, and making the appropriate intermediate definitions to carry out the contraction.

\subsection{Organization}
We prove our main technical results, Proposition~\ref{prop:x_n-decreasing} and Theorem~\ref{thm:exponent}, in section~\ref{sec: main thm}. We apply the contraction method to prove Proposition~\ref{prop: contraction} in section~\ref{sec: contraction}. Finally, we deduce our central limit theorems, Corollary~\ref{cor:CLT-fringe} and Corollary~\ref{cor:CLT-length}, in sections~\ref{sec: clt1} and~\ref{sec: clt2} respectively. 

\subsection*{Acknowledgements}
We thank Svante Janson for pointing out an error in an earlier version of the paper and for referring us to the work of Kolesnik~\cite{kolesnik2025critical}. We thank an anonymous reviewer for a very careful reading that allowed us to substantially improve the paper. LLMs were used only for copyediting and proofreading in the subsequent revision. AB is supported by NSF GRFP 2141064 and NSERC. BC is supported by NSF GRFP 2141064 and Simons Investigator Award 622132 to EM. EM is supported by Vannevar Bush Faculty Fellowship ONR-N00014-20-1-2826 and Simons Investigator Award 622132. 

\section{Proof of the main technical theorem}\label{sec: main thm}
In this section we prove our main results about the recursion~\eqref{eq:recursion}.

\subsection{Monotonicity}
We begin with the more general Proposition~\ref{prop:x_n-decreasing}. The proof proceeds by rewriting the recursion in terms of the consecutive differences $(x_n - x_{n+1})_{n \in \N}$. This nearly gives what we want, as all of the terms in the new recursion are positive except one. We get around this by taking a linear combination of recursions to eliminate the negative term. It then suffices to check that the other coefficients in the linear combination remain positive, relying on the imposed monotonicity condition.

\begin{proof}[Proof of Proposition~\ref{prop:x_n-decreasing}]
    Let $s(n,i) = \sum_{j=1}^i p(n,j)$ for $1 \leq i \leq n-1$. Using summation by parts, we can write that for $n > k$, 
    \begin{equation}\label{eq:x_n-decomp}
        x_n = x_{n-1} + \sum_{i=1}^{n-2} s(n, i)(x_i-x_{i+1}).
    \end{equation}
    Notice that the terms corresponding to $i=1, \ldots, k-2$ in the sum are all 0 due to the initial conditions. We can thus isolate the $(x_{k-1}-x_k)$ term to write 
    \[ 0 = (x_{n-1}-x_n)\frac{1}{s(n, k-1)} + \sum_{i=k}^{n-2}(x_i-x_{i+1})\frac{s(n, i)}{s(n, k-1)} + (x_{k-1}-x_k) \]
    for $n > k$, and 
    \[ 0 = (x_n-x_{n+1})\frac{1}{s(n+1, k-1)} + \sum_{i=k}^{n-1}(x_i-x_{i+1})\frac{s(n+1, i)}{s(n+1, k-1)} + (x_{k-1}-x_k) \]
    for all $n \geq k$. Subtracting the first from the second of the above equations and letting 
    \begin{equation}\label{eq:alpha-def}
        \alpha_{i,n} = \frac{s(n+1, i)}{s(n+1, k-1)} - \frac{s(n,i)}{s(n, k-1)}        
    \end{equation}
    for each $k \leq i \leq n-1$, we have 
    \begin{equation}\label{eq:exact-diff}
        0 = (x_n - x_{n+1})s(n+1, k-1)^{-1} + \sum_{i=k}^{n-1}(x_i - x_{i+1})\alpha_{i,n}.
    \end{equation}
    Assuming that $\alpha_{i,n} < 0$ we can conclude as follows. Inductively, if $x_i - x_{i+1} > 0$ for each $k \leq i \leq n-1$, we deduce that $x_n-x_{n+1} > 0$. Note that the base case $x_k > x_{k+1}$ holds if $k \ge 2$. Therefore, we must have $x_n-x_{n+1} > 0$ for each $n \geq k$, completing the proof. 

    It therefore remains to prove that $\alpha_{i,n} < 0$ for each $k \leq i \leq n-1$. This is equivalent to 
    \[ s(n+1, i)s(n, k-1) < s(n,i)s(n+1, k-1). \]
    Both sides of the inequality when expanded have $i(k-1)$ terms. The difference in corresponding terms (left minus right hand side) is 
    \[ p(n+1,x)p(n,y) - p(n,x)p(n+1, y) \]
    for $1\leq x \leq i-1$ and $1\leq y \leq k-2$. Notice that for $1 \leq a,b\leq k-2$ the difference for the $x=a$ and $y=b$ term cancels exactly with the difference for the $x=b$ and $y=a$ term. Here we rely on the fact that $i \geq k$ so that all terms where $y \ge x$ are cancelled. The remaining differences are for $x > y$, which are negative since we assumed 
    \[ \frac{p(n, x)}{p(n+1, x)} \geq \frac{p(n,y)}{p(n+1, y)} \]
    for all $x > y$. This proves $\alpha_{i,n} < 0$. 
\end{proof}

\begin{proof}[Proof of Theorem~\ref{thm:exponent}\ref{thm:x_n-decreasing}]
    We check the conditions of Proposition~\ref{prop:x_n-decreasing} for $p(n,i) = \frac{1}{h_{n-1}(n-i)}$. It is clear that $\sum_{i=1}^{n-1} p(n,i) = 1$. It turns out to be simpler (and advantageous for future calculations) to check by definition that $\alpha_{i,n} < 0$ directly. By the proof above, it suffices to show that
    \[ p(n+1,x)p(n,y) - p(n,x)p(n+1,y) < 0 \]
    for $x > y$. In this case we can write explicitly that the left hand side is a (positive) constant multiple of 
    \begin{equation}\label{eq: diff}
        \frac{1}{n + 1 -x}\cdot \frac{1}{n-y} - \frac{1}{n-x}\cdot \frac{1}{n + 1 -y} = \frac{y-x}{(n + 1-x)(n-x)(n + 1-y)(n-y)}
    \end{equation}
    which is clearly negative for $x > y$.   
    \end{proof}

    \subsection{Long-term error estimate}
    We next prove~\ref{thm:x_n-simple-rate}. The idea is to notice that the single negative term from the proof of Proposition~\ref{prop:x_n-decreasing} must control all of the positive terms. This gives a bound on the rate of decay of the tail end of the sequence. 
    \begin{proof}[Proof of Theorem~\ref{thm:exponent}\ref{thm:x_n-simple-rate}]
    We return to \eqref{eq:x_n-decomp}. Note that in this setting, $s(n,i) = \frac{h_{n-1} - h_{n-i-1}}{h_{n-1}}$. For all $n > k$, adding the ($i=n-1$)-term $(h_{n-1}-h_0)(x_{n-1}-x_n) = -h_{n-1}(x_n - x_{n-1})$ to both sides of \eqref{eq:x_n-decomp} gives
    \[
    \sum_{i=1}^{n-1} (h_{n-1}-h_{n-i-1}) (x_i - x_{i+1}) = 0,\]
    that is, the negative contribution from $x_{k-1} - x_k = -x_k$ exactly balances all the other positive terms:
    \[\sum_{i=k}^{n-1} (h_{n-1}-h_{n-i-1}) (x_i - x_{i+1}) = (h_{n-1} - h_{n-k})x_k \leq \frac{k}{n-k} x_k .\]
    For $\lceil n/2 \rceil \leq i \leq n-1$, we have $h_{n-1} - h_{n-i-1} \geq h_{n-1} - h_{\lfloor n/2\rfloor - 1} \geq \log 2$, so for any $n$ such that $\lceil n/2 \rceil \geq k$, we can take terms $i = \lceil n/2\rceil, \dots, n-1$ of the sum on the left hand side to obtain
    \[
    x_{\lceil n/2 \rceil} - x_n \leq \frac{1}{\log 2} \frac{k}{n-k} x_k.
    \]
    Writing $x_n - x = (x_n - x_{2n}) + (x_{2n} - x_{4n}) + \cdots $ and using $2^jn-k \geq 2^j(n-k)$ proves the statement for $n \geq 2k$, while for $k < n < 2k$ it holds trivially as $x_n - x \leq x_k \leq \frac{kx_k}{n-k}$.
\end{proof}

\subsection{Short-term error estimate}
To prove~\ref{thm:bootstrap}, we require bounds on the initial terms of the sequence. We leverage the fact that the contributions arising from the first positive term $x_k$ are decaying quickly to inductively show that the sequence itself must be decaying quickly as well. This induction reaches a bottleneck after the first $k^{1/2-\epsilon}$ terms, but we bootstrap the argument to extend it to the first $k^{1-\epsilon}$ terms at the cost of constant factors.
\begin{proof}[Proof of Theorem~\ref{thm:exponent}\ref{thm:bootstrap}]
    We proceed by double induction on $\ell$ and $m$. For the base case, we fix $\ell=1$ and prove 
    \begin{equation}\label{eq:pre-bootstrap}
        x_{k+m} \leq \frac{cx_k}{mh_k} \ \text{ for all } \ m \leq k^{1/2 - 1/(c+1)}.
    \end{equation}
    For the base case $m=1$, the recursive equation gives $x_{k+1} = \frac{x_k}{h_k} < \frac{cx_k}{h_k}$. Proceeding by induction over $m$,
    \begin{align*}
        x_{k+m+1} &= \sum_{i=1}^{k+m} \frac{1}{h_{k+m}(k+m+1-i)} \cdot x_i \\
        &\leq \frac{x_k}{(m+1)h_k} + \sum_{i=1}^{m} \frac{1}{h_k(m+1-i)} \cdot \frac{cx_k}{ih_k} \\
        &= \frac{x_k}{(m+1)h_k} + \frac{cx_k}{h_k^2}\sum_{i=1}^m \frac{1}{i(m+1-i)} \\
        &= \frac{x_k}{(m+1)h_k} + \frac{cx_k}{h_k^2}\cdot \frac{2h_{m}}{m+1} \\
        &\leq \frac{cx_k}{(m+1)h_k}
    \end{align*}
    where the last inequality holds as long as $2ch_m \leq (c-1)h_k$, which is satisfied for $m \leq k^{1/2 - 1/(c+1)}$ and $k$ sufficiently large. This proves \eqref{eq:pre-bootstrap}. 

    We now proceed with the induction on $\ell$. 
    Suppose we have shown the statement for some $\ell \in \N$. We prove the statement for $\ell+1$ via induction on $m$:
    \begin{align*}
        x_{k+m+1} &= \sum_{i=1}^{k+m} \frac{1}{h_{k+m}(k+m+1-i)}\cdot x_i \\
        &\leq \frac{x_k}{(m+1)h_k} + \sum_{i=1}^{k^{d_\ell}} \frac{1}{h_k(m+1-i)} \cdot \frac{c^\ell x_k}{ih_k} + \sum_{i=k^{d_\ell}}^m \frac{1}{h_k(m+1-i)}\cdot \frac{c^{\ell+1}x_k}{ih_k} \\
        &= \frac{x_k}{(m+1)h_k}\left( 1 + \frac{c^\ell(h_{k^{d_\ell}} + h_m - h_{m-k^{d_\ell}})}{h_k} + \frac{c^{\ell+1}(h_{m-k^{d_\ell}} + h_m - h_{k^{d_\ell}})}{h_k} \right) \\
        &\leq \frac{x_k}{(m+)h_k}\left(1 + \frac{2c^{\ell+1}h_m - (c^{\ell+1}-c^\ell)h_{k^{d_\ell}}}{h_k}\right) \\
        &< \frac{c^{\ell+1}x_k}{(m+1)h_k},
    \end{align*}
    where the last inequality holds as long as $2c^{\ell+1}h_m - (c^{\ell+1}-c^\ell)h_{k^{d_\ell}} < (c^{\ell+1}-1)h_k$. For large $k$, this condition reduces to
    \[ d_{\ell+1} < \frac{1}{2} - \frac{1}{2c^{\ell+1}} + \frac{c-1}{2c}d_\ell. \]
    By the inductive hypothesis that $d_\ell = 1 - \frac{1}{c+1} - 2^{-\ell}$ we can take
    \begin{align*}
        d_{\ell+1} &< \frac{1}{2} - \frac{1}{2c^{\ell+1}} + \frac{c-1}{2c}\left(1 - \frac{1}{c+1} - 2^{-\ell}\right) \\
        &= \frac{1}{2} - \frac{1}{2c^{\ell+1}} + \frac{1}{2} - \frac{1}{c+1} - \frac{c-1}{2c}2^{-\ell} \\
        &= 1 - \frac{1}{c+1} - 2^{-(\ell+1)} - \frac{1}{2c^{\ell+1}} + \frac{1}{2^{\ell+1}c};
    \end{align*}
    in particular $d_{\ell+1} = 1 - \frac{1}{c+1} - 2^{-(\ell+1)}$ as long as $c > 2$, since then $\frac{1}{2^{\ell+1}c} > \frac{1}{2c^{\ell+1}}$..
\end{proof}

\subsection{Asymptotic exponent of error}\label{subsec:exponent-proof}
Finally, to prove the last two parts we rely on an integral approximation to precisely estimate \eqref{eq:exact-diff}. We give near exact expressions for the coefficients $\alpha_{i,n}$, estimate the sum via multiple applications of the Euler--Maclaurin formula (in Lemma~\ref{lem: EM} below), and inductively prove tighter upper and lower bounds for the consecutive differences. Summing the differences over the tail of the sequence yields the desired estimates. These steps turn out to be sharp in the exponent for $k$ fixed.
\begin{proof}[Proof of Theorem~\ref{thm:exponent}\ref{thm:exponent-fixed-k}--\ref{thm:exponent-growing-k}]
    For each $n \geq k$, let $\delta_n \ceq x_n - x_{n+1}$. Then, from \eqref{eq:exact-diff}, we have
    \[
    \delta_n = s(n+1, k-1) \sum_{i=k}^{n-1} (-\alpha_{i,n}) \delta_i = \frac{h_n - h_{n-k+1}}{h_n} \sum_{i=k}^{n-1} (-\alpha_{i,n}) \delta_i. 
    \]
    Using \eqref{eq: diff} to expand $\alpha_{i,n}$ as given in \eqref{eq:alpha-def}, we can write 
    \[ 
        -\alpha_{i,n} = \sum_{a=k-1}^{i-1}\sum_{b=0}^{k-2} \frac{a-b}{(n-a)(n-a-1)(n-b)(n-b-1)} \cdot \frac{1}{(h_n-h_{n-k+1})(h_{n-1}-h_{n-k})}. 
    \]
    Let $\alpha_{i,n}' = (h_n - h_{n-k+1}) (-\alpha_{i,n})$ such that
    \begin{equation}\label{eq:d_n-exact}
        \delta_n = \frac{1}{h_n}\sum_{i=k}^{n-1} \alpha_{i,n}' \delta_i.
    \end{equation}
    Using that $\sum_{b=x}^{y} \frac{1}{(n-b)(n-b-1)} = \sum_{b=x}^{y} \big(\frac{1}{n-b-1} - \frac{1}{n-b} \big) 
    = \frac{y-x+1}{(n-x)(n-y-1)}$ for any $x < y$
    and $\frac{b}{(n-b)(n-b-1)} = \frac{n}{(n-b)(n-b-1)} - \frac{1}{n-b-1}$,
    we can write for each $k \leq i \leq n-1$ that
    \[ (h_n-h_{n-k+1})(h_{n-1}-h_{n-k})(-\alpha_{i,n}) := A - B,
    \]
    where 
    \[
    \begin{aligned}
        A &= \sum_{a=k-1}^{i-1} \frac{a}{(n-a)(n-a-1)}\sum_{b=0}^{k-2}\frac{1}{(n-b)(n-b-1)} \\
        &= \bpar{\frac{n(i-k+1)}{(n-i)(n-k+1)} - \sum_{a=k-1}^{i-1} \frac{1}{n-a-1}} \frac{k-1}{n(n-k+1)}
    \end{aligned}
    \]
    and 
    \[
    \begin{aligned}
        B &= \sum_{a=k-1}^{i-1}\frac{1}{(n-a)(n-a-1)} \sum_{b=0}^{k-2}\frac{b}{(n-b)(n-b-1)} \\
        &= \frac{i-k+1}{(n-i)(n-k+1)} \bpar{\frac{n(k-1)}{n(n-k+1)} - \sum_{b=0}^{k-2} \frac{1}{n-b-1}}.
    \end{aligned}
    \]
    Noting the cancellation between the first terms of $A$ and $B$, we end up with 
    \begin{equation*}
        A-B = \frac{i-k+1}{(n-i)(n-k+1)} (h_{n-1} - h_{n-k}) - (h_{n-k} - h_{n-i-1}) \frac{k-1}{n(n-k+1)},
    \end{equation*}
    giving the exact expression 
    \begin{equation}
        \alpha_{i,n}' = \frac{i-k+1}{(n-i)(n-k+1)} - \frac{k-1}{n(n-k+1)} \frac{1}{h_{n-1} - h_{n-k}} (h_{n-k} - h_{n-i-1}).
    \end{equation}
    Using $\frac{k-1}{n-1}\leq h_{n-1} - h_{n-k} \leq \frac{k-1}{n-k+1}$, we obtain the pointwise bounds $\underline\alpha_{i,n} \leq \alpha_{i,n}' \leq \overline\alpha_{i,n}$, where
    \begin{equation}\label{eq:alpha-bounds}
    \begin{aligned}
        &
         \overline\alpha_{i,n} \ceq \frac{1}{n}\bpar{\frac{n}{n-k+1}\frac{i-k+1}{n-i} - (h_{n-k}-h_{n-i-1})}, \ \text{and}\\
        &\underline\alpha_{i,n} \ceq \frac{1}{n-k+1}\bpar{\frac{i-k+1}{n-i} - (h_{n-k}-h_{n-i-1})}. 
    \end{aligned}
    \end{equation}
    Both are nonnegative, since $h_{n-k}-h_{n-i-1} \leq \frac{i-k+1}{n-i}$, so we can bound $\delta_n$ from above and below by
    \begin{align*}
        \delta_n &\leq \frac{1}{h_n} \frac{1}{n-k+1} \sum_{i=k}^{n-1} \bpar{ \frac{i-k+1}{n-i} - \frac{n-k+1}{n} (h_{n-k} - h_{n-i-1}) } \delta_i
        =  \frac{1}{h_n} \sum_{i=k}^{n-1} \overline\alpha_{i,n}\delta_i
    \end{align*}
    and
    \[\begin{aligned}
        \delta_n &\geq \frac{1}{h_n} \frac{1}{n-k+1} \sum_{i=k}^{n-1} \bpar{ \frac{i-k+1}{n-i} - \frac{n-1}{n} (h_{n-k} - h_{n-i-1})} \delta_i \\
        &\geq \frac{1}{h_n} \sum_{i=k}^{n-1} \underline\alpha_{i,n}\delta_i .
    \end{aligned}\]
    Assume the scaling $\delta_n \sim n^{-\gamma}$. For the sake of induction, we hope to show that 
    \begin{align}
        \frac{1}{n}
        \frac{1}{h_n}\sum_{i=k}^{n-1} \left(\frac{n(i-k+1)}{(n-i)(n-k+1)} - (h_{n-k} - h_{n-i-1})\right) i^{-\gamma} &\leq n^{-\gamma}, \nonumber \\ 
        \text{i.e.,}\quad h_n - \frac{1}{n} \sum_{i=k}^{n-1} \left(\frac{n(i-k+1)}{(n-i)(n-k+1)} - (h_{n-k} - h_{n-i-1})\right) \left(\frac{n}{i}\right)^\gamma &\geq 0
        \label{eq:to-show-leq-1} 
    \end{align} 
    for $\gamma < \gamma_*+1$ and 
    \begin{align}
        \frac{1}{n-k+1}
        \frac{1}{h_n}\sum_{i=k}^{n-1} \left(\frac{i-k+1}{n-i} - (h_{n-k} - h_{n-i-1})\right) i^{-\gamma} &\geq n^{-\gamma}, \nonumber \\
        \text{i.e.,}\quad h_n - \frac{1}{n-k+1} \sum_{i=k}^{n-1} \left(\frac{i-k+1}{n-i} - (h_{n-k} - h_{n-i-1})\right) \left(\frac{n}{i}\right)^\gamma &\leq 0 \label{eq:to-show-geq-1}
    \end{align}
    for $\gamma > \gamma_*+1$. The first lines above read $\frac{1}{h_n}\sum_{i=k}^{n-1} \overline\alpha_{i,n} i^{-\gamma} \leq n^{-\gamma}$ and $\frac{1}{h_n}\sum_{i=k}^{n-1} \underline\alpha_{i,n} i^{-\gamma} \geq n^{-\gamma}$, respectively.
    The sums in \eqref{eq:to-show-leq-1} and \eqref{eq:to-show-geq-1} resemble each other, so we adjust them at a cost of $o(1)$
    to be the same. We remark that an error of $o(1)$ rather than $o(h_n)$ is crucial as it will turn out that the left hand sides are of constant order. We show that both sums can be written as 
    \begin{equation} \label{eq:riemann-sum}
    \sum_{i=k}^{n-1} \left(\frac{i/n}{1-i/n} - \log\left(1 + \frac{i/n}{1-i/n}\right)\right)\left(\frac{n}{i}\right)^\gamma \cdot \frac{1}{n} 
    \ + \ o(1).
    \end{equation}
    For \eqref{eq:to-show-leq-1}, for each $k \leq i \leq n-1$ consider the difference 
    \[ \left(\frac{n(i-k+1)}{(n-i)(n-k+1)} - (h_{n-k} - h_{n-i-1})\right) - \left( \frac{i}{n-i} - \log\left(\frac{n}{n-i}\right)\right), \]
    which equals 
    \[ \log\left(\frac{n}{n-i}\right) - \frac{k-1}{n-k+1} - \frac{1}{n-k} - \cdots - \frac{1}{n-i}. \]
    By integral approximation, we know that $\log(\frac{n}{n-i})$ differs from $\frac{1}{n-1} + \cdots + \frac{1}{n-i}$ by at most $\frac{1}{n-i} - \frac{1}{n} = \frac{i}{n(n-i)}$. Thus, this difference is at most $\frac{i}{n(n-i)} + \frac{k^2}{n^2}$. Summing over $i$ from $k$ to $n-1$, the total difference over all terms between the sum in \eqref{eq:to-show-leq-1} and the term in \eqref{eq:riemann-sum} is at most
    \begin{align*}
        \sum_{i=k}^{n-1} \left(\frac{i}{n(n-i)} + \frac{k^2}{n^2}\right)\left(\frac{n}{i}\right)^\gamma \cdot \frac{1}{n} &= \frac{1}{n}\sum_{i=k}^{n-1}\frac{1}{n-i}\left(\frac{n}{i}\right)^{\gamma-1} + \frac{k^2}{n^{3-\gamma}} \sum_{i=k}^{n-1} i^{-\gamma} \\
        &\leq \frac{k^{2-\gamma}}{n^{2-\gamma}}\sum_{i=k}^{n-1} \frac{1}{(n-i)i} + C\frac{k^{3-\gamma}}{n^{3-\gamma}} \\
        &\leq 2 h_n \frac{k^{2-\gamma}}{n^{3-\gamma}} + C\frac{k^{3-\gamma}}{n^{3-\gamma}}
    \end{align*}
    where $C = C(\gamma)$, using $i^{1-\gamma} \leq k^{2-\gamma}i^{-1}$ for $i \geq k$ (which requires $\gamma \geq 2$) and $\sum_{i \geq k} i^{-\gamma} \leq Ck^{1-\gamma}$. This is $o(1)$ as long as $2 \leq \gamma < 3$ and $k = o(n)$. Similarly, the total difference over all terms of 
    \[
    \bpar{ \frac{i-k+1}{n-i} - (h_{n-k} - h_{n-i-1})} - \bpar{ \frac{i}{n-i} - \log \Big(\frac{n}{n-i} \Big)}
    \]
    summed in \eqref{eq:to-show-geq-1} is $o(1)$ as long as $k = o({n}/{\log^{1/(3-\gamma)} n})$, as each difference is at most $\frac{ki}{n(n-i)}$. Thus we can replace the sums in both \eqref{eq:to-show-leq-1} and \eqref{eq:to-show-geq-1} with \eqref{eq:riemann-sum}. 
    Let \[j_n \ceq h_n - \sum_{i=k}^{n-1} \left(\frac{i/n}{1-i/n} - \log\left(1 + \frac{i/n}{1-i/n}\right)\right)\left(\frac{n}{i}\right)^\gamma \cdot \frac{1}{n}.\] 
    Lemma~\ref{lem: EM} below confirms that for sufficiently large $n$, \eqref{eq:to-show-leq-1} and \eqref{eq:to-show-geq-1} are equivalent to
    \[ \sum_{i=1}^\infty \left(\frac{1}{i} - \frac{i}{(i+1)(i+2-\gamma)}\right) > 0 \ \ad \ \sum_{i=1}^\infty \left(\frac{1}{i} - \frac{i}{(i+1)(i+2-\gamma)}\right) < 0, \]
    which hold respectively for $\gamma<\gamma_*+1$ and $\gamma>\gamma_*+1$.
    So we may proceed with the induction. 
    
    Fix $\epsilon > 0$ arbitrary. For $n_0 = n_0(\epsilon)$ sufficiently large, we have due to \eqref{eq:to-show-leq-1} that
    \[ \frac{1}{h_n} \sum_{i=k}^{n-1} i^{-\gamma_*-1 + \epsilon} \overline\alpha_{i,n} \leq n^{-\gamma_* -1 + \epsilon}  \]
    for all $n \geq n_0$. Choose $C = C(\epsilon)$ large enough so that $\delta_n \leq Cn^{-\gamma_* -1+ \epsilon}$ for all $k \leq n < n_0$. By induction, using \eqref{eq:to-show-leq-1} in \eqref{eq:d_n-exact} together with $\alpha_{i,n}' \leq \overline\alpha_{i,n}$ from \eqref{eq:alpha-bounds}, we deduce that
    \begin{align*}
        \delta_n = \frac{1}{h_n} \sum_{i=k}^{n-1} \alpha_{i,n}' \delta_i
        &\leq \frac{1}{h_n} \sum_{i=k}^{n-1} \overline\alpha_{i,n}\, Ci^{-\gamma_*-1 + \epsilon}
        \\
        &\leq Cn^{-\gamma_* -1 + \epsilon}
    \end{align*}
    for $n \in \N$. Since $\epsilon > 0$ was arbitrary, this implies $\delta_n \leq n^{-\gamma_* -1+ o(1)}$. 

    This gives only the upper bound $C(\epsilon) = O(1)$ for $k$ fixed. To obtain the upper bound on $C(\epsilon, k)$ in item \ref{thm:exponent-growing-k}, we truncate the sequence such that the induction base case is between $2k$ and $m = \min\{k\log k, \sqrt{nk}\}$. By Theorem~\ref{thm:exponent}\ref{thm:bootstrap} we know that for $C = O(x_k(\log k)^{\gamma_* + 1 -\epsilon}k^{0.01-\epsilon}) = O(x_kk^{0.01})$ and all $2k \leq n \leq m$,
    \[ \delta_n = O\bpar{\frac{x_k}{k^{0.99}}} \leq Ck^{\gamma_*}n^{-\gamma_*-1+\epsilon} .\]
    Thus, it suffices to show that the truncated contribution is negligible, namely
    \[ \sum_{i=k}^{2k} i^{-\gamma_*-1+\epsilon}\alpha_{i,n}'\frac{1}{h_n} = o(n^{-\gamma_*-1+\epsilon}). \]
    For $k \leq i \ll n$ we have, using again $\alpha_{i,n}' \leq \overline\alpha_{i,n}$,
    \begin{align*}
        \alpha_{i,n}' &\leq \frac{i-k+1}{(n-i)(n-k+1)} - \frac{1}{n}\left(\frac{1}{n-k} + \cdots + \frac{1}{n-i}\right) \\
        &= \sum_{\ell=k}^i \left(\frac{1}{(n-i)(n-k+1)} -\frac{1}{n(n-\ell)}\right) \\
        &= \sum_{\ell=k}^i \frac{n(i+k-\ell-1) + i - ik}{n(n-\ell)(n-i)(n-k+1)} \\
        &\leq \frac{(i+k-1)(i-k+1)}{(n-i)^3}.
    \end{align*}
    Using this bound, we deduce
    \begin{align*}
        \frac{1}{h_n}\sum_{i=k}^{2k} i^{-\gamma_*-1+\epsilon}\alpha_{i,n}' &\leq \frac{1}{h_n}\sum_{i=k}^{2k} \frac{i^2}{n^3}\cdot i^{-\gamma_*-1+\epsilon} = \frac{1}{n^3h_n}\sum_{i=k}^{2k} i^{1-\gamma_*+\epsilon} \leq \frac{(2k)^{2-\gamma_*+\epsilon}}{n^3h_n} \\ 
        &= o(n^{-\gamma_*-1+\epsilon}).
    \end{align*}
    
    Conversely, for the lower bound, set $n_1 = n_1(\epsilon)$ sufficiently large such that \eqref{eq:to-show-geq-1} gives
    \[ \frac{1}{h_n} \sum_{i=k}^{n-1} i^{-\gamma_* -1 - \epsilon} \underline\alpha_{i,n} \geq n^{-\gamma_* -1- \epsilon}
    \]
    for all $n \geq n_1$. Choose $c = c(\epsilon)$ small enough so that $\delta_n \geq cn^{-\gamma_*-1 - \epsilon}$ for all $k \leq n < n_1$. By induction, using \eqref{eq:to-show-geq-1} together with $\alpha_{i,n}' \geq \underline\alpha_{i,n}$, we deduce that $\delta_n \geq cn^{-\gamma_*-1 - \epsilon}$ for all $n$. This implies the lower bound $\delta_n \geq n^{-\gamma_*-1 - o(1)}$. 

    To conclude, note that
    \[ x_n - x = \sum_{i=n}^\infty \delta_i = \sum_{i=n}^\infty i^{-\gamma_*-1+o(1)} = n^{-\gamma_* + o(1)}. \]
    The conclusion for~\ref{thm:exponent-growing-k} follows similarly.
\end{proof}
\begin{lem}\label{lem: EM}
    For $2 < \gamma < 3$ and $k = o(n)$, we have that $j_n$ can be written as
    \[ h_n - \sum_{i=k}^{n-1} \left(\frac{i/n}{1-i/n} - \log\left(1 + \frac{i/n}{1-i/n}\right)\right)\left(\frac{n}{i}\right)^\gamma \cdot \frac{1}{n} = \sum_{i=1}^\infty \left(\frac{1}{i}-\frac{i}{(i+1)(i+2-\gamma)}\right) + o(1).  \]
\end{lem}
\begin{proof}
    We use the Euler--Maclaurin formula in the following form (see e.g.~\cite{whittaker2021}): for $a, b \in \N$ and $f: \R \to \R$ a continuous function, 
    \begin{equation}\label{eq: EM}
        \sum_{i=a}^b f(i) = \int_a^b f(x) \, dx + \frac{f(a)+f(b)}{2} + R_1,
    \end{equation}
    with 
    \[ \absd{R_1} \leq \int_a^b |f'(x)| \, dx. \]
    We apply this to the function 
    \[ f(i) = \frac{1}{n-i} - \left(\frac{i/n}{1-i/n} - \log\left(1 + \frac{i/n}{1-i/n}\right)\right)\left(\frac{n}{i}\right)^\gamma \cdot \frac{1}{n} . \]
    We also define 
    \[ g(x) = \frac{1}{1-x} - \frac{x^{1-\gamma}}{1-x} - x^{-\gamma}\log(1-x) , \]
    and note that $f(x) = g(x/n)/n$ and thus $f'(x) = g'(x/n)/n^2$. We compute each term in \eqref{eq: EM}. 
    \begin{align*}
        \int_k^{n-1} f(x) \, dx &= \int_{k/n}^{1-1/n} g(x) \, dx.
    \end{align*}
    Taylor expanding within the integral, we get 
    \begin{align*}
        \int_{k/n}^{1-1/n} g(x) \, dx &= \int_{k/n}^{1-1/n} \sum_{i=0}^\infty x^i - \sum_{i=1}^\infty x^{-\gamma + i} + \sum_{i=1}^\infty \frac{1}{i}x^{-\gamma + i} \, dx \\
        &= \int_{k/n}^{1-1/n} \sum_{i=0}^\infty x^i - \sum_{i=2}^\infty \frac{i-1}{i}x^{-\gamma + i} \, dx \\
        &= \left[\sum_{i=1}^\infty \frac{1}{i}x^i - \sum_{i=2}^\infty \frac{i-1}{i(i+1-\gamma)}x^{i+1-\gamma}\right]_{k/n}^{1-1/n} \\
        &= \sum_{i=1}^\infty \left(\frac{1}{i} - \frac{i}{(i+1)(i+2-\gamma)}\right)\left(1-\frac{1}{n}\right)^i + O((k/n)^{3-\gamma}).
    \end{align*}
    Notice that since the sum $\sum_{i=1}^\infty \left(\frac{1}{i} - \frac{i}{(i+1)(i+2-\gamma)}\right)$ is convergent, we have 
    \begin{align*}
        \sum_{i=1}^\infty \left(\frac{1}{i} - \frac{i}{(i+1)(i+2-\gamma)}\right)\left(1-\frac{1}{n}\right)^i &= (1+o(1))\sum_{i=1}^{\log n} \left(\frac{1}{i} - \frac{i}{(i+1)(i+2-\gamma)}\right) \\
        &\qquad + \sum_{i=\log n}^\infty \left(\frac{1}{i} - \frac{i}{(i+1)(i+2-\gamma)}\right)\left(1-\frac{1}{n}\right)^i \\
        &= \sum_{i=1}^\infty \left(\frac{1}{i} - \frac{i}{(i+1)(i+2-\gamma)}\right) + o(1).
    \end{align*}
    
    For the middle term of \eqref{eq: EM} we have 
    \begin{align*}
        \frac{f(k)+f(n-1)}{2} &= \frac{1}{2}\left( \frac{1}{n-k} - \frac{(n/k)^{\gamma-1}}{n-k} - \frac{n^{\gamma-1}}{k^\gamma}\log\left(1-\frac{k}{n}\right)\right) \\
        &\qquad + \frac{1}{2}\left(1-\left(\frac{n-1}{n}-\frac{\log n}{n}\right)\left(\frac{n}{n-1}\right)^\gamma \right) \\
        &=\left( -\frac{1}{4} + o(1)\right)k^{2-\gamma}n^{\gamma-3} + O\bpar{\frac{\log n}{n}} = o(1).
    \end{align*}
    The final equality holds from Taylor expanding $\log(1-k/n)$, and the $O(\log n/n)$ term comes from expanding the term $\big(\frac{n}{n-1}\big)^\gamma = 1 + \frac{\gamma}{n} + O(n^{-2})$ on the second line.
    
    Finally, to handle the error term in \eqref{eq: EM} we compute 
    \[ g'(x) = \frac{-(\gamma+1)x^{1-\gamma}+1+\gamma x^{-\gamma}}{(1-x)^2} + \gamma x^{-\gamma-1}\log(1-x). \]
    We claim (see appendix~\ref{appendix}) that $g'(x) > 0$ on the interval $(0,1)$. Thus, 
    \begin{align*} |R_1| \leq \int_k^{n-1} |f'(x)| \, dx &= \int_k^{n-1} f'(x) \, dx = f(n-1) - f(k) \\
    &= \left(\frac{1}{4} + o(1)\right)k^{2-\gamma}n^{\gamma-3} + O\bpar{\frac{\log n}{n}} = o(1), 
    \end{align*}
    computed similarly as above. 

    To conclude, note that for $k=o(n)$, we have
    \[ \sum_{i=k}^{n-1} f(i) = h_{n-k} - \sum_{i=k}^{n-1} \left(\frac{i/n}{1-i/n} - \log\left(1 + \frac{i/n}{1-i/n}\right)\right)\left(\frac{n}{i}\right)^\gamma \cdot \frac{1}{n} \]
    and $h_{n-k} = h_n + o(1)$. 
\end{proof}

\section{Contraction in the critical beta-splitting tree}\label{sec: contraction}
In this section we work within the framework of~\cite{NR:04} for degenerate distributional recursions and prove Proposition~\ref{prop: contraction}. We will bound the distance from the Gaussian distribution in the Zolotarev metric $\zeta_3$ defined as 
\[ \zeta_3(X,Y) = \sup_{f \in \mathcal{F}_3} \absd{\E{f(X) - f(Y)}}
\]
where $\mathcal{F}_3 = \{f \in C^2(\R, \R): |f''(x) - f''(y)| \leq |x-y|\}$ is the space of twice differentiable functions with 1-Lipschitz second derivative. For background on this metric, we refer to~\cite{zolotarev1977approximation, zolotarev1978ideal, rachev1991probability}. We use the following nice properties of this metric: for any random variables $X$, $Y$ and $Z$ independent of $X$ and $Y$,
\begin{enumerate}
    \item convergence in $\zeta_3$ implies weak convergence,
    \item $\zeta_3(X+Z, Y+Z) \leq \zeta_3(X, Y)$,
    \item $\zeta_3(cX, cY) = |c|^3 \zeta_3(X, Y)$. 
\end{enumerate}

We begin with a series of definitions, chosen to align with~\cite{NR:04} as best possible. For every $n \in \N$, we abbreviate $\mu_n \ceq \E{X_n}$, and define the normalized random variable $Z_n = \frac{X_n - \E{X_n}}{\sigma \sqrt{n}}$ with $\tau_n \ceq \sqrt{\va[Z_n]} = 1 + O(n^{-\epsilon})$, and write $Z_n'$ and $Z_n''$ for the same normalizations of $X_n'$ and $X_n''$. Our distributional recurrence is equivalent to
\[ Z_n \disteq \sqrt{\frac{I_n}{n}} Z'_{I_n} + \sqrt{\frac{n-I_n}{n}} Z''_{n-I_n} + b^{(n)} \]
where $b^{(n)} = \frac{1}{\sigma\sqrt{n}}(b_n + \mu_{I_n} + \mu_{n-I_n} - \mu_n)$. Taking expectations in the recursion of Proposition~\ref{prop: contraction} gives $\mu_n = \E{\mu_{I_n} + \mu_{n-I_n} + b_n}$, so that
\begin{equation}\label{eq:b-mean-zero}
    \E{b^{(n)}} = 0,
\end{equation}
which we use repeatedly below. Let $N_1$ and $N_2$ be independent (from everything) standard Gaussians and define
\[ Z_n^* \ceq \sqrt{\frac{I_n}{n}} \tau_{I_n}N_1 + \sqrt{\frac{n-I_n}{n}} \tau_{n-I_n} N_2 + b^{(n)}. \]
$Z_n^*$ will serve as an intermediary between $Z_n$ and the standard Gaussian.

The proof of Proposition~\ref{prop: contraction} is adapted from the strategy of \cite[Theorem 2.1]{NR:04}. We present it here in a self-contained manner, while isolating the parts of the proof specific to our setting and postponing the remainder to the appendix. The strategy consists of two main steps, captured by the following two lemmas. 

The first provides a recursive upper bound on the distance of our recursively defined variable $Z_n$ from a Gaussian. This is analogous to~\cite[(17)--(18)]{NR:04} and its proof is in appendix~\ref{app: lem1}.
\begin{lem}\label{lem: app1}
    Let $d_n = \zeta_3(Z_n, \tau_n N_1)$. Then 
    \begin{equation}\label{eq:d_n-recursion}
    d_n \leq \E{\left(\frac{n-I_n}{n}\right)^{3/2}d_{n-I_n} + \left(\frac{I_n}{n}\right)^{3/2}d_{I_n}} + \zeta_3(Z_n^*, \tau_n N_1).         
    \end{equation} 
\end{lem}

The second bounds the extra additive term in the above recursion \eqref{eq:d_n-recursion}. We first require a few more definitions. Let $G_n \ceq \sqrt{\frac{I_n}{n}\tau_{I_n}^2 + \frac{n-I_n}{n}\tau_{n-I_n}^2}$ and $\Delta_n \ceq \sqrt{|G_n^2 - \tau_n^2|}$. The following is analogous to~\cite[(23)]{NR:04} and its proof is in appendix~\ref{app: lem2}.
\begin{lem}\label{lem: app2}
    There exists a constant $C > 0$ such that
    \[ \zeta_3(Z_n^*, \tau_n N_1) \leq C(|\tau_n-1|^3 + \|\Delta_n\|_3^3  + \|b^{(n)}\|_3^3 + \|G_n-1\|_3^3 + \|b^{(n)}\|_2(|\tau_n-1| + \|G_n-1\|_2)). \]
\end{lem}

\begin{proof}[Proof of Proposition~\ref{prop: contraction}]
    We begin by bounding each of these terms in Lemma~\ref{lem: app2}. Throughout this proof, $C$ is a universal constant that is allowed to change from line to line. We may assume $\epsilon < 1/2$, as decreasing $\epsilon$ only weakens the hypotheses.
    \begin{itemize}
        \item By assumption,
        \[ |\tau_n-1|^3 \le Cn^{-3\epsilon}. \]
        \item Similarly, by assumption on moments of $b_n$ and on $\E{X_n}$ we have
            \begin{align*}
                \|b^{(n)}\|_3^3 &= \frac{1}{\sigma^3n^{3/2}}\E{|b_n + \mu_{I_n} + \mu_{n-I_n} - \mu_n|^3} \\
                &\leq \frac{1}{\sigma^3n^{3/2}}\E{(|b_n| + Cn^{1/2-\epsilon})^3} \\
                &\leq Cn^{-\epsilon},
            \end{align*}
            where the dominant contribution is $\E{|b_n|^3}/n^{3/2} \leq n^{-\epsilon}$.
        \item For the $G_n$ term, where the variances below are conditional on $I_n$,
        \begin{align*}
            \|G_n - 1\|_3^3
            &= \E{\left| \sqrt{\frac{I_n}{n} \frac{\va[X'_{I_n} \mid I_n]}{\sigma^2 I_n} + \frac{n-I_n}{n} \frac{\va[X''_{n-I_n}\mid I_n]}{\sigma^2 (n-I_n)} } - 1  \right|^3} \\
            &= n^{-3/2} \E{\left| \sigma^{-1} \sqrt{\va[X'_{I_n}\mid I_n] + \va[X''_{n-I_n}\mid I_n]} - \sqrt{n}\right|^3 } \\
            &\leq n^{-3/2} \sum_{i=\lceil n/2 \rceil}^{n-1} \frac{n}{i(n-i)h_{n-1}} \left|\sqrt{i + n-i} + Cn^{1/2-\epsilon} - \sqrt{n} \right|^3 \\
            &\leq Cn^{-3\epsilon}.
        \end{align*}
        \item Finally,
        \begin{align*}
            \|\Delta_n\|_3 &= \|\sqrt{\absd{\tau_n^2 - G_n^2}}\|_3 \leq \sqrt{\|\tau_n^2 - G_n^2\|_3} \leq \sqrt{|\tau_n^2-1| + \|G_n^2-1\|_3}.
        \end{align*}
        Notice that both $\tau_n$ and $G_n$ are bounded, so this is bounded by \\ $C\sqrt{|\tau_n-1| + \|G_n-1\|_3}$ which we have already estimated.
    \end{itemize}
    All together, this shows that there exists a constant $C_1 > 0$ such that
    \[ \zeta_3(Z_n^*, \tau_n N_1) \leq C_1n^{-\epsilon/2}. \]
    Combined with the above estimates, Lemmas~\ref{lem: app1} and~\ref{lem: app2} imply that
    \begin{align*}
        d_n &\leq \E{\left(\frac{n-I_n}{n}\right)^{3/2}d_{n-I_n} + \left(\frac{I_n}{n}\right)^{3/2}d_{I_n}} + C_1n^{-\epsilon/2}.
    \end{align*}
    Fix $n_0$ sufficiently large and let $C_0 \geq \max\{d_nn^{\epsilon/100}: n \leq n_0\}$, both to be specified later. By construction, $d_n \leq C_0n^{-\epsilon/100}$ for all $n \leq n_0$. We now inductively show that $d_n \leq C_0 n^{-\epsilon/100}$ for all $n > n_0$. Assuming the inequality to hold for all $i < n$, we have
    \begin{align*}
        d_n &\leq \sum_{i=\lceil n/2 \rceil}^{n-1} \frac{n}{i(n-i)h_{n-1}} \left( \left(\frac{n-i}{n}\right)^{3/2}d_{n-i} + \left(\frac{i}{n}\right)^{3/2}d_i\right) + C_1n^{-\epsilon/2} \\
        &\leq \sum_{i=1}^{n-1} \frac{n}{i(n-i)h_{n-1}}\left(\frac{i}{n}\right)^{3/2}\frac{C_0}{i^{\epsilon/100}} + C_1n^{-\epsilon/2} \\
        &\leq \frac{C_0}{n^{\epsilon/100}} \sum_{i=1}^{n-1} \frac{i^{1/2}}{n^{1/2}(n-i)h_{n-1}} \frac{n^{\epsilon/100}}{i^{\epsilon/100}} + C_1n^{-\epsilon/2} \\
        &\leq \frac{C_0}{n^{\epsilon/100}}\cdot \frac{1}{h_{n-1}}\sum_{i=1}^{n-1} \frac{1}{n-i}\bpar{\frac{i}{n}}^{\beta} + C_1n^{-\epsilon/2},
    \end{align*}
    where $\beta \ceq \frac12 - \frac{\epsilon}{100} > 0.49$. Substituting $j = n-i$ and using the concavity bound $(1-u)^\beta \leq 1 - \beta u$, valid for $0 \leq u \leq 1$, the last sum is
    \[ \sum_{i=1}^{n-1} \frac{1}{n-i}\bpar{\frac in}^\beta = \sum_{j=1}^{n-1} \frac{(1-j/n)^\beta}{j} \leq \sum_{j=1}^{n-1} \frac{1 - \beta j/n}{j} = h_{n-1} - \beta\frac{n-1}{n} \leq h_{n-1} - \frac{\beta}{2}. \]
    Since $h_{n-1} \leq 2\log n$ for $n \geq 3$, this gives $\frac{1}{h_{n-1}}\sum_{i=1}^{n-1}\frac{1}{n-i}(\frac in)^\beta \leq 1 - \frac{c}{\log n}$ with $c \ceq 1/10$.
    Now choose $C_0 = \max\{ \max\{d_nn^{\epsilon/100}: n \leq n_0\}, \max\{\frac{C_1\log n}{cn^{49\epsilon/100}}: n > n_0\}\}$. Then
    \[
    d_n \leq \frac{C_0}{n^{\epsilon/100}} \bpar{1 - \frac{c}{\log n}} + C_1n^{-\epsilon/2} = \frac{C_0}{n^{\epsilon/100}} \bpar{1 - \frac{c}{\log n} + \frac{C_1}{C_0n^{49\epsilon/100}}}  \leq \frac{C_0}{n^{\epsilon/100}}.
    \]
    This completes the induction step and proves that
    $d_n \leq {C_0}{n^{-\epsilon/100}}$
    for all $n \in \N$ as desired, which is sufficient to conclude, as
    \begin{equation*}
    \zeta_3\left(\frac{X_n-\E{X_n}}{\sqrt{\va{X_n}}}, N_1\right) = \tau_n^{-3}\zeta_3(Z_n, \tau_n N_1) = (1+o(1)) d_n = o(1).  \qedhere
    \end{equation*}
\end{proof}

\section{CLT for the number of copies of a clade}\label{sec: clt1}
Fix $k \in \N$. For our setting of the count of number of fringe subtrees, we can write, for $n \geq 2$,
\begin{equation}\label{eq:N_n-k-distributional}
    N_n(k) \disteq N'_{L_n}(k) + N''_{n - L_n}(k) + \ind{n=k},
\end{equation}
where $(N_m'(k))_{m \geq 1}$ and $(N_m''(k))_{m \geq 1}$ are independent copies of $(N_m(k))_{m \geq 1}$, independent also of $L_n$. As $q_n(\cdot)$ is symmetric we may replace $L_n$ by $I_n$, so this falls into the setting of Proposition~\ref{prop: contraction} with $b_n = \ind{n=k}$, and thus it suffices to check $\E{N_n(k)}$ and $\va[N_n(k)]$ to deduce our central limit theorem. 

We use the recursive estimates to analyze the mean and variance of $N_n(k)$, the number of fringe subtrees of size $k$ of the critical beta-splitting tree DTCS($n$). 
Denote $E_n = E_n(k) = \frac{1}{n}\E{N_n(k)}$ and $V_n = V_n(k) = \frac{1}{n}\va[N_n(k)]$.
\begin{lem}\label{lem: expectation}
    There exists $\mu \ceq \mu(k)$ and a universal constant C such that $|E_n(k) - \mu(k)| \leq C/{n}$ for all $k = n^{o(1)}$. 
\end{lem}
\begin{proof} 
    Recall from \eqref{eq:total-expectation} that the law of total expectation yields that for $n > k$,
    \[ E_n = \sum_{i=1}^{n-1} \frac{1}{h_{n-1}(n-i)}\cdot E_i.\]
    It is easy to see that 
    $E_1 = \cdots E_{k-1} = 0$ and $E_k = 1/k$.
    From  Theorem~\ref{thm:exponent}\ref{thm:x_n-simple-rate}, there exists $\mu > 0$ such that $E_n - \mu \le \frac{2}{\log 2}\frac{1}{n-k}$. 
\end{proof}

\begin{lem}\label{lem: variance}
    There exists $\sigma \ceq \sigma(k)$ and a universal constant C such that $|V_n(k) - \sigma(k)^2| \leq Cn^{-1/6}$ for all $k = n^{o(1)}$.
\end{lem}
\begin{proof}
    By the law of total variance, for $n > k$, where the inner variance and expectation are conditional on $L_n$,
    \begin{multline*}
        V_n = \frac{1}{n}\va[N_n(k)] = \\ \frac{1}{n}\E{\va[N'_{L_n}(k) + N''_{n-L_n}(k) \mid L_n]} + \frac{1}{n}\va\!\left[\E{N'_{L_n}(k) + N''_{n-L_n}(k) \mid L_n}\right].
    \end{multline*}
    By the same calculation as in the proof of Lemma~\ref{lem: expectation}, using conditional independence given $L_n$,
    \[ \frac{1}{n}\E{\va[N'_{L_n}(k) + N''_{n-L_n}(k) \mid L_n]} = \sum_{i=1}^{n-1} \frac{1}{h_{n-1}(n-i)}\cdot V_i. \]
    On the other hand, by Lemma~\ref{lem: expectation}, we know that
    \[ \epsilon_n \vcentcolon= \frac{1}{n}\va[\E{N'_{L_n}(k) + N''_{n-L_n}(k) \mid L_n}] = O\left(\frac{1}{n}\right). \]
    Thus, we can write
    \[ V_n = \sum_{i=1}^{n-1} \frac{1}{h_{n-1}(n-i)} \cdot V_i + \epsilon_n. \]
    Since $N_n(k)$ is deterministic for $n \leq k$, we have $V_1 = \cdots = V_k = 0$, so the sequence is driven entirely by the increments $\epsilon_n$. For each $m > k$, define the sequence $(V_n^{(m)})_{n \geq 1}$ by
    \[ V_1^{(m)} = \cdots = V_{m-1}^{(m)} = 0, \quad V_m^{(m)} = \epsilon_m, \quad V_n^{(m)} = \sum_{i=1}^{n-1}\frac{1}{h_{n-1}(n-i)}V_i^{(m)} \ \text{ for } n > m. \]
    As both sides satisfy the same recursion with the same initial conditions,
    \[ V_n = \sum_{m=k+1}^n V_n^{(m)}. \]
    By Theorem~\ref{thm:exponent}\ref{thm:x_n-decreasing}--\ref{thm:x_n-simple-rate}, there exist constants $\sigma_{k+1} \ceq \sigma_{k+1}(k)$, $\sigma_{k+2} \ceq \sigma_{k+2}(k), \ldots$ such that
    \begin{align*}
        V_n^{(m)} = \sigma_m^2 + O\left(\frac{1}{n-m}\right).
    \end{align*}
    We use this for $m < \sqrt{n}$. Notice error term in this bound is very crude for $m$ large. Indeed, by Theorem~\ref{thm:exponent}\ref{thm:bootstrap} (with $c = 3$ and $\ell=1$) we know that $0 \leq \sigma_m^2 \leq V_n^{(m)} \leq \frac{1}{m^{7/6}}$ when $n-m > m^{1/6}$. Thus for $\sqrt{n} \leq m < n - n^{1/6}$ we have 
    \[ V_n^{(m)} = \sigma_m^2 + O\left(\frac{1}{m^{7/6}}\right) \] 
    Finally, for $m$ very close to $n$, i.e. $n - n^{1/6} \leq m \leq n$ we have the naive bound
    \[ V_n^{(m)} = \sigma_m^2 + O\left(\frac{1}{m}\right) \]  
    following from $0 \leq \sigma_m^2 \leq V_n^{(m)} \leq \frac{1}{m}$. All together this gives
    \[ V_n = \sum_{m=k+1}^n \sigma_m^2 + O(n^{-1/6}). \]
    Moreover, by Lemma~\ref{lem: expectation}, we know $\epsilon_m = O\left({1}/{m}\right)$
    so Theorem~\ref{thm:exponent}\ref{thm:bootstrap} tells us that $\sigma_m^2 = O(m^{-7/6})$, so that
    \[ \sum_{m=k+1}^\infty \sigma_m^2 < \infty \text{ and } \sum_{m=n}^\infty \sigma_m^2 = O(n^{-1/6}).  \]
    Thus, taking $\sigma^2 = \sum_{m=k+1}^\infty \sigma_m^2$ yields
    \[ V_n = \sigma^2 + O(n^{-1/6}). \qedhere \]
\end{proof}

\begin{rem}\label{rem: positive variance}
    The limiting variance is strictly positive. Indeed, by linearity and \eqref{eq:a_k}, the recursion with initial data $x_m = v$ converges to $v\,a(m)$, so that $\sigma_m^2 = \epsilon_m a(m)$ and $\sigma^2 = \sum_{m>k} \epsilon_m a(m)$. It thus suffices to exhibit one $m$ with $\epsilon_m > 0$, and $m = 2k$ works: the conditional expectation $\E{N'_{L}(k) + N''_{2k-L}(k) \mid L}$ equals $2$ when $L = k$, but equals $\E{N_{2k-1}(k)} \leq 1$ when $L = 1$, as a tree with $2k-1$ leaves cannot contain two disjoint clades of size $k$. The same holds for a fixed clade $\chi$ of size $j$: since $\E{N_m(\chi)} = p_\chi \E{N_m(j)}$, where $p_\chi > 0$ is the probability that DTCS$(j)$ equals $\chi$, we have $\epsilon_m(\chi) = p_\chi^2 \epsilon_m(j)$.
\end{rem}

Note that, fixing a clade $\chi$ of size $j = |\chi|$, the same distributional recursion \eqref{eq:N_n-k-distributional} holds for $N_n(\chi)$, the number of times $\chi$ appears as a fringe tree in DTCS$(n)$, now with $b_j \sim \Ber(p_\chi)$ and $b_n = 0$ for $n \neq j$. Lemmas~\ref{lem: expectation} and \ref{lem: variance} also both still hold, with different constants $\mu \ceq \mu(\chi)$ and $\sigma \ceq \sigma(\chi)$. So the following proof is identical for $N_n(\chi)$.

\begin{proof}[Proof of Corollary~\ref{cor:CLT-fringe}]
    Follows from Proposition~\ref{prop: contraction} and Lemmas~\ref{lem: expectation} and~\ref{lem: variance}.
\end{proof}

\section{CLT for the total length of the CTCS}\label{sec: clt2}
In this section we prove our central limit theorem for $\Lambda_n$. Recall that $\Lambda_n$ is the total length of CTCS$(n)$. In order to apply Proposition~\ref{prop: contraction} we rely on the distributional recurrence
\begin{equation}
\Lambda_n \disteq \Lambda'_{L_n} + \Lambda''_{n-L_n} + \mathrm{Exp}(h_{n-1})   \quad \text{for } n \ge 2,
\end{equation}
where $(\Lambda'_n)_{n \geq 1}$ and $(\Lambda''_n)_{n \geq 1}$ are independent copies of $(\Lambda_n)_{n \geq 1}$ and the $\mathrm{Exp}(h_{n-1})$ variable is independent of everything else; as in section~\ref{sec: clt1}, we may replace $L_n$ by $I_n$.
Moreover, since each time a clade of size $i$ splits, the length increases by an independent Exp$(h_{i-1})$ amount of time, the total length can be computed as 
\[ \Lambda_n \disteq \sum_{i=2}^n \sum_{j=1}^{N_n(i)} \mathrm{Exp}(h_{i-1}), \]
where each exponential random variable is independent. 
We use these distributional equalities to compute the mean and variance of $\Lambda_n$. 

\begin{lem}\label{lem: length expectation}
    There exists a constant $\mu_\Lambda$ such that $\E{\Lambda_n} = \mu_\Lambda n + O\left(n^{0.41}\right)$. 
\end{lem}
\begin{proof}
    The expected value of $\Lambda_n$ can be computed as 
    \[ \E{\Lambda_n} = n\sum_{k=2}^n \frac{\E{N_n(k)}}{n} \cdot \frac{1}{h_{k-1}} \]
    as each clade with $k$ leaves contributes an independent $\mathrm{Exp}(h_{k-1})$ amount of length. We know that $\E{N_n(k)} = \mu_k n + O(1)$ by Lemma~\ref{lem: expectation} but we will require more precise estimates on the error term $O(1)$ for various $k$. 
    
    For $k < n^{0.6}$ we use the bound $(\frac{k}{n})^{\gamma_*}k^{-0.99}$ from Theorem~\ref{thm:exponent}\ref{thm:exponent-growing-k}, for $n^{0.6} < k < n - n^{0.99}$ we use the bound ${k^{-1.99}}$ from Theorem~\ref{thm:exponent}\ref{thm:bootstrap} and for $k > n-n^{0.99}$ we use the bound $({k(n-k)})^{-1}$ also from Theorem~\ref{thm:exponent}\ref{thm:bootstrap}; see Remark~\ref{remark}. 
    Recalling the exact form of the limit $\E{N_k(k)}/n = a(n,k)/k \to {6h_{k-1}}/({\pi^2 (k-1)})$ from \eqref{eq:a_k}, we have the estimate
    \begin{align*}
        \E{\Lambda_n} &= n \sum_{k=2}^{n} \frac{a(k)}{k} \frac{1}{h_{k-1}} + n \sum_{k=2}^n \frac{a(n,k) - a(k)}{k} \frac{1}{h_{k-1}} \\
        &= n \sum_{k=2}^n \frac{6}{\pi^2 k (k-1)} + n\left(\sum_{k=2}^{n^{0.6}}  \left(\frac{k}{n}\right)^{\gamma_* +o(1)}k^{-0.99}  \frac{1}{h_{k-1}}  + \sum_{k=n^{0.6}}^{n-n^{0.99}} k^{-1.99}  \frac{1}{h_{k-1}}  \right. \\
        & \left. \hspace{4.5cm} +  \sum_{k=n-n^{0.99}}^n (k(n-k))^{-1} \frac{1}{h_{k-1}}  \right) \\
        &\leq \frac{6}{\pi^2}n + n^{1-\gamma_*+0.95} + n(n^{0.6})^{-0.99} + h_n \\
        &= \frac{6}{\pi^2}n + O(n^{0.41}). \qedhere
    \end{align*}
\end{proof}

\begin{lem}\label{lem: length variance}
    There exists a constant $\sigma_\Lambda$ such that $\va[\Lambda_n] = \sigma_\Lambda^2n + O(n^{0.9})$.
\end{lem}
\begin{proof}
    We compute the variance recursively using the law of total variance as in Lemma~\ref{lem: variance}. This yields the decomposition 
    \begin{align*}
      \frac{1}{n}\va[\Lambda_n] &= \sum_{i=1}^{n-1} \frac{1}{h_{n-1}(n-i)} \frac{\va[\Lambda_i]}{i} + \frac{1}{nh_{n-1}^2} + n^{-0.18}
    \end{align*}
    where the $n^{-0.18}$ arises from Lemma~\ref{lem: length expectation}. 
    
    Recall that adding additional terms at step $k$ of the recursion corresponds to initiating a new recursion with $x_k = {1}/({kh_{k-1}^2}) + k^{-0.18}$. Since we know that such a recursion converges to $\sigma^{(k)} = O(k^{-1.17})$ by Theorem~\ref{thm:exponent}\ref{thm:bootstrap}, the main contributions are summable and yield our desired constant $\sigma^2_\Lambda = \sum_{k=2}^\infty \sigma^{(k)}$. Thus it remains to control the error terms on these new recursions. By the same error bounds and decomposition as Lemma~\ref{lem: length expectation}, the error term is bounded by \vspace{-.5\bs}
    \[ n^{-\gamma_*+2.44} + n(n^{0.6})^{-0.17} + n\sum_{i=n-n^{0.9}}^{n-1} \frac{1}{i^{0.18}(n-i)}\cdot \frac{1}{h_{i-1}} = O(n^{0.9}). \qedhere \]
\end{proof}

\begin{proof}[Proof of Corollary~\ref{cor:CLT-length}]
    Follows from Proposition~\ref{prop: contraction}, Lemmas~\ref{lem: length expectation} and~\ref{lem: length variance}, and \\ $\E{\mathrm{Exp}(h_{n-1})^3} = {6} / {h_{n-1}^3}$.
\end{proof}

\bibliographystyle{abbrv}
\bibliography{biblio}

\appendix 

\section{Positivity of $g'(x)$ in section~\texorpdfstring{\ref{subsec:exponent-proof}}{2.4}}\label{appendix}
Recall that we have 
\[ g'(x) = \frac{-(\gamma+1)x^{1-\gamma}+1+\gamma x^{-\gamma}}{(1-x)^2} + \gamma x^{-\gamma-1}\log(1-x). \]
Showing $g'(x) > 0$ is equivalent to 
\[ -(\gamma+1)x^{1-\gamma} + 1 + \gamma x^{-\gamma} + \gamma x^{-\gamma-1}(1-x)^2\log(1-x) > 0. \]
Multiplying both sides by $x^{\gamma+1}$ we obtain the inequality
\begin{equation}\label{eq: derivative}
    -(\gamma+1)x^2 + x^{\gamma+1} + \gamma x + \gamma(1-x)^2\log(1-x) > 0.
\end{equation}

We first treat the $(1-x)^2\log(1-x)$ term. Taking derivatives we find 
\begin{align*}
    \frac{d}{dx} (1-x)^2\log(1-x) &= x-1-2(1-x)\log(1-x), \\
    \frac{d^2}{dx^2} (1-x)^2\log(1-x) &= 3+2\log(1-x),  \\
    \frac{d^j}{dx^j} (1-x)^2\log(1-x) &= (-1)^{j-1}\frac{2(j-3)!}{(x-1)^{j-2}}, \qquad \text{for $j \geq 3$.}  
\end{align*}
This implies that the Taylor expansion is 
\[ (1-x)^2\log(1-x) = -x + \frac{3}{2}x^2 - \frac{1}{3}x^3 - x^4\sum_{j=1}^\infty \frac{2}{(j+3)(j+2)(j+1)}x^{j-1}. \]
For $0 < x < 1$, we have
\begin{align*}
    \sum_{j=1}^\infty \frac{2}{(j+3)(j+2)(j+1)}x^{j-1} &\leq \sum_{j=1}^\infty \frac{2}{(j+3)(j+2)(j+1)} 
    = \frac{1}{6}.
\end{align*}
This implies that for all $0 < x < 1$,
\[ (1-x)^2\log(1-x) \geq -x + \frac{3}{2}x^2 - \frac{1}{3}x^3 - \frac{1}{6}x^4. \]

Plugging this into \eqref{eq: derivative}, it suffices to show that 
\[ -(\gamma+1)x^2 + x^{\gamma+1} +\gamma x + \gamma\left(-x + \frac{3}{2}x^2 - \frac{1}{3}x^3 - \frac{1}{6}x^4\right) > 0. \]
This is equivalent to 
\[ h(x) \ceq \left(\frac{\gamma}{2}-1\right) - \frac{\gamma}{3}x - \frac{\gamma}{6}x^2 + x^{\gamma-1} > 0 \]
for all $0 < x < 1$. Taking a derivative we find 
\begin{align*}
    h'(x) &= -\frac{\gamma}{3}-\frac{\gamma}{3}x + (\gamma-1)x^{\gamma-2}
\end{align*}
Since $2 < \gamma < 3$ and $0 < x < 1$, we have $h(0)>0$ and $h(1)=0$. In order to verify $h(x) > 0$ for all $0 < x < 1$, it suffices to check that $h(r) > 0$ whenever $h'(r) = 0$ and $0 < r < 1$. This is because if $h$ had another root, then the conditions $h(0) > 0$ and $h(1) = 0$ force $h$ to have a point at which $h(r) \leq 0$ and $h'(r) = 0$, deriving a contradiction. 

The condition $h'(r) = 0$ restricts us to $(\gamma-1)r^{\gamma-2} = \gamma/3+r\gamma/3$. Substituting this into the expression for $h$, we get 
\[ h(r) = \left(\frac{\gamma}{2}-1\right) - \frac{\gamma}{3}r - \frac{\gamma}{6}r^2 + \frac{r\gamma/3 + r^2\gamma/3}{\gamma-1}. \]
Since multiplying through by $6(\gamma-1)$ does not change the sign, it suffices to evaluate
\begin{align*}
    6(\gamma-1)h(r) &= 3\gamma^2-9\gamma+6 - 2\gamma(\gamma-2)r + \gamma(3-\gamma)r^2.
\end{align*}
This quadratic is minimized at $m = \frac{2\gamma(\gamma-2)}{2\gamma(3-\gamma)} = \frac{\gamma-2}{3-\gamma}$. Notice that for $\gamma > 2.5$, we have $m > 1$. Moreover, at $r=1$, the quadratic evaluates to $6-2\gamma > 0$. This implies that $h(r) > 0$ whenever $r < 1$ for $\gamma > 2.5$.

From now on, we may assume $\gamma \leq 2.5$. We continue by evaluating the quadratic at $m$. We want to show that
\[ 6(\gamma-1)h(r) \geq (3\gamma-6)(\gamma-1) - 2\frac{\gamma(\gamma-2)^2}{3-\gamma} + \frac{\gamma(\gamma-2)^2}{3-\gamma} > 0. \]
This is equivalent to 
\[ 3(3-\gamma)(\gamma-1) > \gamma(\gamma-2) \iff 0 > 4\gamma^2-14\gamma +9. \]
It is easy to check that this is true for all $2 < \gamma \leq 2.5$, so we conclude that $h(r) > 0$ whenever $h'(r) = 0$ and $0 < r < 1$.

\section{Proof of Lemma~\texorpdfstring{\ref{lem: app1}}{3.1}}\label{app: lem1}
By the triangle inequality, 
\[ \zeta_3(Z_n, \tau_n N_1) \le \zeta_3(Z_n, Z_n^*) + \zeta_3(Z_n^*, \tau_n N_1). \]
Because of the strategically similar definitions, we can express
\begin{align*}
    &\zeta_3(Z_n, Z_n^*) \\
    &= \E{\E{\zeta_3\left(\sqrt{\frac{I_n}{n}} Z'_{I_n} + \sqrt{\frac{n-I_n}{n}} Z''_{n-I_n} + b^{(n)}, \sqrt{\frac{I_n}{n}} \tau_{I_n}N_1 + \sqrt{\frac{n-I_n}{n}} \tau_{n-I_n}N_2 + b^{(n)}\right)\middle\vert I_n}} \\
    &\leq \E{\E{\zeta_3\left(\sqrt{\frac{I_n}{n}} Z'_{I_n} + \sqrt{\frac{n-I_n}{n}} Z''_{n-I_n}, \sqrt{\frac{I_n}{n}} \tau_{I_n}N_1 + \sqrt{\frac{n-I_n}{n}} \tau_{n-I_n}N_2\right)\middle\vert I_n}} \\
    &\leq \E{\E{\zeta_3\left(\sqrt{\frac{I_n}{n}} Z'_{I_n} + \sqrt{\frac{n-I_n}{n}} Z''_{n-I_n}, \sqrt{\frac{I_n}{n}} Z'_{I_n} + \sqrt{\frac{n-I_n}{n}} \tau_{n-I_n}N_2\right)\middle\vert I_n}} \\
    &\qquad + \E{\E{\zeta_3\left(\sqrt{\frac{I_n}{n}} Z'_{I_n} + \sqrt{\frac{n-I_n}{n}} \tau_{n-I_n}N_2, \sqrt{\frac{I_n}{n}} \tau_{I_n}N_1 + \sqrt{\frac{n-I_n}{n}} \tau_{n-I_n}N_2\right)\middle\vert I_n}} \\
    &\leq \E{\E{\zeta_3\left(\sqrt{\frac{n-I_n}{n}} Z''_{n-I_n}, \sqrt{\frac{n-I_n}{n}} \tau_{n-I_n}N_2\right)\middle\vert I_n}} \\
    &\qquad + \E{\E{\zeta_3\left(\sqrt{\frac{I_n}{n}} Z'_{I_n}, \sqrt{\frac{I_n}{n}} \tau_{I_n}N_1\right)\middle\vert I_n}} \\
    &= \E{\E{\left(\frac{n-I_n}{n}\right)^{3/2}d_{n-I_n} + \left(\frac{I_n}{n}\right)^{3/2}d_{I_n}\middle\vert I_n}} \\
    &= \E{\left(\frac{n-I_n}{n}\right)^{3/2}d_{n-I_n} + \left(\frac{I_n}{n}\right)^{3/2}d_{I_n}}.
\end{align*}
Here we used the properties $\zeta_3(X+Z, Y+Z) \leq \zeta_3(X, Y)$ and $\zeta_3(cX, cY) = |c|^3\zeta_3(X, Y)$ with the triangle inequality. In the first inequality, we remove the common summand $b^{(n)}$ as $Z'_{I_n}$, $Z''_{n-I_n}$, $N_1$ and $N_2$ are all conditionally independent of $b^{(n)}$ given $I_n$.

\section{Proof of Lemma~\texorpdfstring{\ref{lem: app2}}{3.2}}\label{app: lem2}
Recall that $G_n = \sqrt{\frac{I_n}{n}\tau_{I_n}^2 + \frac{n-I_n}{n}\tau_{n-I_n}^2}$ so that $Z_n^* \disteq G_n \Tilde{N} + b^{(n)}$ where $\Tilde{N}$ is an independent standard Gaussian, and $\Delta_n = \sqrt{|G_n^2 - \tau_n^2|}$. Let $A = \{G_n > \tau_n\}$. Then, we have the following distributional equalities:
letting $N$ and $N'$ be two more independent (from everything) standard Gaussians, 
    \begin{align*}
        Z_n^* &\disteq \ind{A}(\tau_n N + \Delta_n N') + \ind{A^c}(G_n N) +  b^{(n)} =: \widehat{Z_n^*}, \\
        \tau_n N_1 &\disteq \ind{A}(\tau_n N) + \ind{A^c}(G_n N + \Delta_n N') =: \widehat{N_1}.
    \end{align*}
    Essentially, this isolates the ``shared'' parts of $Z_n^*$ and $\tau_n N_1$ and writes the ``unshared'' parts as an independent (orthogonal) Gaussian. Taylor expanding around this ``shared'' Gaussian $N$ and using properties of test functions $f \in \mathcal{F}_3$, we can write 
    \begin{align*} \E{f(Z_n^*) - f(\tau_n N_1)} &= \E{f(\widehat Z^*_n) - f(\widehat N_1)} \\
      &= \E{S_1f'(N) + S_2 \frac{f''(N)}{2} + R(\widehat{Z_n^*}, N) + R(\widehat{N_1}, N)}, 
    \end{align*}
    where
    \[S_1 = (\widehat{Z_n^*} - N) - (\widehat{N_1} - N) \quad\text{and}\quad S_2 = (\widehat{Z_n^*} - N)^2 - (\widehat{N_1} - N)^2.
    \]

    For $S_1$, the terms multiplying the first derivative, we have $(\widehat{Z_n^*} - N) - (\widehat{N_1} - N) = \widehat{Z_n^*} - \widehat{N_1}$. By independence, any term that depends only linearly on $N'$, or $b^{(n)}$ vanishes in expectation (here we use \eqref{eq:b-mean-zero}). This implies $\E{S_1 f'(N)} = 0$.

    For $S_2$, the terms multiplying the second derivative, we have
    \begin{align*}
        (\widehat{Z_n^*} - N)^2 - (\widehat{N_1} - N)^2 &= (\widehat{Z_n^*} - \widehat{N_1})(\widehat{Z_n^*} + \widehat{N_1} - 2N) \\
        &= \left(b^{(n)} + \Delta_n N'(\ind{A}-\ind{A^c})\right) \\
        &\qquad \left( 2\tau_n N\ind{A} + 2G_n N \ind{A^c} - 2N + \Delta_n N' +  b^{(n)} \right)
    \end{align*}
    We distribute the terms in the first parentheses and take expectations against $f''(N)$. Since $N'$ is independent of $N$ and of $(b^{(n)}, \tau_n, G_n, \Delta_n, \ind{A})$, every term linear in $N'$ has expectation $0$, while the two quadratic terms together contribute $\E{f''(N)}(\E{G_n^2} + \E{(b^{(n)})^2} - \tau_n^2)$. This vanishes, since
    \[ \tau_n^2 = \va[Z_n] = \E{G_n^2 + (b^{(n)})^2}. \]
    The only surviving contribution is the one in which $b^{(n)}$ multiplies $2N(\ind{A}(\tau_n-1) + \ind{A^c}(G_n-1))$, so that
    \[ \E{S_2  f''(N)/2} = \E{f''(N)N b^{(n)}(\ind{A}(\tau_n-1)+\ind{A^c}(G_n-1))}. \]
    As computed in \cite[p.~2847]{NR:04}, $|\E{f''(N)N}| \leq 1$ so by the Cauchy--Schwarz inequality,
    \[ |\E{S_2  f''(N)/2}| \leq \|b^{(n)}\|_2(|\tau_n-1| + \|G_n-1\|_2). \]
    
    $R(\widehat{Z_n^*}, N)$ and $R(\widehat{N_1}, N)$ are the remainder terms from the Taylor expansion. By the 1-Lipschitz property of $f''$ we have 
    \begin{align*}
        \E{R(\widehat{Z_n^*}, N)} &\leq \E{|\widehat{Z_n^*} - N|^3}
    \end{align*}
    In expanding the third power, terms that involve both $\ind{A}$ and $\ind{A^c}$ vanish, so we are left only with 
    \begin{align*}
        \E{R(\widehat{Z_n^*}, N)} &\leq \E{\left|N(\tau_n-1) + \Delta_n N' +  b^{(n)}\right|^3} + \E{\left|N(G_n-1) + b^{(n)}\right|^3} \\
        &\leq C \big( |\tau_n-1|^3 + \|\Delta_n\|_3^3 + \|b^{(n)}\|_3^3 + \|G_n-1\|_3^3 \big)
    \end{align*}
    for some constant $C$. Noting that $\E{R(\widehat{N_1}, N)}$ is bounded by the same term, we have proven the lemma. 
\end{document}